\documentclass[a4paper,10pt]{article}
\usepackage{test,natbib,graphicx,enumerate}

\title{An Information-Theoretic Approach to Nonparametric Estimation, Model Selection, and Goodness of Fit}
\author{Alexis Akira Toda \thanks{Department of Economics, Yale University.  Email: \href{mailto:alexisakira.toda@yale.edu}{alexisakira.toda@yale.edu}} \thanks{I am deeply indebted to my thesis advisor, Donald Brown.  The financial supports from the Cowles Foundation, the Nakajima Foundation, and Yale University are greatly acknowledged.}}
\date{This Version: \today}

\numberwithin{equation}{section}

\newcommand{\diff}{\mathrm{d}}
\renewcommand{\e}{\mathrm{e}}
\renewcommand{\i}{\mathrm{i}}

\begin{document}
\maketitle

\begin{abstract}
This paper applies the recently axiomatized Optimum Information Principle (minimize the Kullback-Leibler information subject to all relevant information) to nonparametric density estimation, which provides a theoretical foundation as well as a computational algorithm for maximum entropy density estimation.  The estimator, called \emph{optimum information estimator}, approximates the true density arbitrarily well.  As a by-product I obtain a measure of goodness of fit of parametric models (both conditional and unconditional) and an \emph{absolute} criterion for model selection, as opposed to other conventional methods such as AIC and BIC which are \emph{relative} measures.
\end{abstract}

\section{Introduction}\label{sec:intro}
This paper applies the \emph{Optimum Information Principle} recently axiomatized by \cite{toda-maxent-bayes} by refining Jaynes's axioms of plausible reasoning \cite[Chapters 1 and 2]{jaynes2003} to nonparametric density estimation.  The optimum information principle, which is more fundamental than Bayesian inference \cite[]{bayes1763}, Fisher's maximum likelihood principle \cite[]{fisher1912},\footnote{Although maximum likelihood is attributed to \cite{fisher1912}, it was already used by Laplace and Gauss a century before.}  Jaynes's maximum entropy principle \cite[]{jaynes1957a}, and Kullback's principle of minimum discrimination information \cite[p.~37]{kullback1959}, prescribes to minimize the information gain (measured by the Kullback-Leibler information \cite[]{kullback-leibler1951}) of updating from a prior to posterior subject to all relevant information.  In particular, for nonparametric density estimation, it means to maximize the Shannon entropy \cite[]{shannon1948} of a density subject to sample moment constraints.  Such an information-theoretic approach to density estimation has been known (see \cite{wu2003} and the references therein) but has not yet gained popularity in econometrics possibly due to the lack of a theoretical foundation as well as a simple algorithm of computation.  This paper provides both.  As a by-product I obtain a measure of goodness of fit of parametric models and a criterion for model selection, which is closely related to BIC \cite[]{schwarz1978} but distinct from other conventional methods such as AIC \cite[]{akaike1974} and BIC in that it is an \emph{absolute} criterion, not relative.  AIC and BIC can only pick the best model among the competing ones, but it may be the case that all models are poor. Our measure tells the poor fit if all models are indeed poor.

\section{Optimum Information Estimator}\label{sec:OIE}
Suppose that $\set{X_n}$ are i.i.d.~random variables taking values in $\R^K$, with an unknown density $f(x)$.  Let $\set{x_n}_{n=1}^N$ be the realization of $\set{X_n}$.  Our task is to obtain a reasonable estimate $\hat{f}$ of $f$ from the data.  Since the data $\set{x_n}$ is a finite set, it is compact and discrete.  Therefore there is no reason to believe that the true distribution $f$ has an unbounded support or that $f$ is discontinuous.  Hence throughout the paper let us assume that $f$ is continuous and compactly supported on $S\subset B(R)$, where $B(R)$ denotes the closed ball with radius $R>\max_n\norm{x_n}$ with center at the origin.

\paragraph{Sample moments}
Sample moments of $f$ are computed by $\hat{m}_i=\frac{1}{N}\sum_{n=1}^Nx_{ni}$, $\hat{m}_{ij}=\frac{1}{N}\sum_{n=1}^Nx_{ni}x_{nj}$, etc.  In general let us introduce the multi-index \cite[p.~236]{folland1999} of nonnegative integers $\alpha=(\alpha_1,\dotsc,\alpha_K)$ and let the $\alpha$-th sample moment denoted by
$$\hat{m}_\alpha=\frac{1}{N}\sum_{n=1}^Nx^\alpha:=\frac{1}{N}\sum_{n=1}^Nx_{n1}^{\alpha_1}\dotsb x_{nK}^{\alpha_K}.$$
We set $\abs{\alpha}=\sum_{k=1}^K\abs{\alpha_k}$, $\alpha!=\alpha_1!\dotsb\alpha_K!$, etc.  Even more generally, if $T:S\to \R^L$ is a Lebesgue measurable function, the sample moments of $T$ can be defined by $\hat{T}:=\frac{1}{N}\sum_{n=1}^NT(x_n)$.  The function $T$ represents the moments that the econometrician considers relevant for inference.  If there is no particular reason to choose otherwise, it is natural to set $T(x)=(x^\alpha)_{\abs{\alpha}\le A}$, where $A$ is typically an even integer from 2 to 10.

\paragraph{Optimum information estimator defined}
The \emph{optimum information estimator} (OIE) of $f$, denoted by $\hat{f}$, is defined by
\begin{align}
&\hat{f}=\argmin_{g\in L^1(S)} \int_S g(x)\log g(x)\diff x\notag\\
&\text{subject to}~\int_S T(x) g(x)\diff x=\hat{T},~\int_S g(x)\diff x=1,\label{eq:1}
\end{align}
where $\diff x$ denotes the Lebesgue measure and $\int_S g\log g$ is the Kullback-Leibler information of the density $g$ with respect to the uniform prior on $S$.\footnote{By convention we set $0\log 0=0$ and $x\log x=\infty$ if $x<0$.}  Hereafter all integrations are carried out on the compact set $S$ and therefore we drop the subscript $S$ from the integral sign.  Minimizing the information as in \eqref{eq:1} is optimal from an information theoretic as well as a Bayesian point of view.  See \cite{toda-maxent-bayes} for a theoretical justification of this definition.

\paragraph{Computing the optimum information estimator}
The minimization problem \eqref{eq:1} is a special case of an entropy-like minimization problem, which can be solved by Fenchel duality.\footnote{For a good account of Fenchel duality, see \cite{rockafellar1970}, \cite{borwein-lewis2006} (finite-dimensional), and \cite{luenberger1969} (infinite-dimensional).}  Since $\int g=1$, \eqref{eq:1} is equivalent to
\begin{align}
&\min_{g\in L^1(S)} \int [g(x)\log g(x)-g(x)]\diff x\notag\\
&\text{subject to}~\int T(x) g(x)\diff x=\hat{T},~\int g(x)\diff x=1.\tag{P}\label{eq:primal}
\end{align}
By Corollary 2.6 and Example 5.6(ii) of \cite{borwein-lewis1991}, the dual problem of \eqref{eq:primal} is given by
\begin{equation}
\max_{z\in\R,\lambda\in\R^L}\left[z+\lambda'\hat{T}-\int \e^{z+\lambda'T(x)}\diff x\right].\tag{D}\label{eq:dual}
\end{equation}
I assume that a regularity condition of the Fenchel duality theorem holds and therefore the dual problem \eqref{eq:dual} has a solution and the primal and dual value coincide.  One such regularity condition is that $\hat{T}$ belongs to the interior of $T(S)=\set{T(x)|x\in S}$, which is very weak.\footnote{See \cite{bot-wanka2006} for a nice discussion on regularity conditions.}  Since the objective function in \eqref{eq:dual} contains an integral over the compact set $S$, it is always finite and $C^\infty$ in $(z,\lambda)$ \cite[Proposition B.5]{toda2010}.  Since $(z,\lambda)$ are unconstrained,
the maximum is obtained by differentiating and setting equal to zero.  Partially differentiating \eqref{eq:dual} with respect to $z$, we obtain
$$1-\int \e^{z+\lambda'T(x)}\diff x=0\iff z=-\log\left(\int \e^{\lambda'T(x)}\diff x\right).$$
Substituting this into \eqref{eq:dual}, and using the Fenchel duality theorem \cite[Corollary 2.6]{borwein-lewis1991}, after some algebra we obtain:
\begin{thm}[Fenchel Duality]\label{thm:dual}
\begin{align}
H_{\min}&:=\text{Minimized Kullback-Leibler information}\notag \\
&=\min_{g\in L^1(S)}\Set{\int g\log g|\int Tg=\hat{T},\int g=1}\notag\\
&=\max_{\lambda\in\R^L}\left[\lambda'\hat{T}-\log\left(\int \e^{\lambda'T(x)}\diff x\right)\right].\label{eq:2.2}
\end{align}
\end{thm}
The duality theorem \ref{thm:dual} was first exploited by \cite{gibbs1902},\footnote{Interestingly, Gibbs was the thesis advisor of Irving Fisher, the first Yale Ph.D.~in economics.} although its full understanding had to wait for the development of modern convex analysis half a century later.  The reduced dual problem \eqref{eq:2.2} has a unique solution $\hat{\lambda}$ if $\dim T(S)=L$ (\ie, $T(S)$ is not contained in a lower dimensional affine space) because in that case the function $\phi(\lambda)=\log\left(\int \e^{\lambda'T(x)}\diff x\right)$ is strictly convex \cite[Proposition B.4]{toda2010}.  In most applications the maximization problem \eqref{eq:2.2} has no closed-form solutions.  However, since the objective function is differentiable and concave \cite[Proposition B.5]{toda2010}, a numerical solution $\hat{\lambda}$ can be easily obtained by the Newton-Raphson algorithm \cite[Chapter 10]{luenberger1969}.

Differentiating the objective function in \eqref{eq:2.2} with respect to $\lambda$ and setting equal to zero, $\hat{\lambda}$ satisfies $\hat{T}=\E_{\hat{f}}[T(X)]=\int T(x)\hat{f}(x)\diff x$,
where
\begin{equation}
\hat{f}(x)=\frac{\e^{\hat{\lambda}'T(x)}}{\int \e^{\hat{\lambda}'T(x)}\diff x}.\label{eq:2.4}
\end{equation}
The function $\hat{f}$ is the optimum information estimator.  To verify this, substitute $\hat{f}$ defined by \eqref{eq:2.4} for $g$ in \eqref{eq:2.2} and we can see that the equality is satisfied.

The functional form in \eqref{eq:2.4} shows that the optimum information estimator $\hat{f}$ belongs to an exponential family, and one might guess that the Lagrange multiplier $\hat{\lambda}$ is the maximum likelihood estimator of that family.  This conjecture is indeed true.

\begin{prop}\label{prop:ML}
$\hat{\lambda}$ is a maximum likelihood estimator\footnote{More precisely, it is a quasi-maximum likelihood (QML, see \cite{huber1967}) estimator because the model is misspecified.} for the exponential family $\set{f(x;\lambda)}_{\lambda\in\R^L}$, where $f(x;\lambda)\propto \e^{\lambda'T(x)}$.
\end{prop}
\begin{proof}
The log likelihood of the model $f_\lambda(x):=\e^{\lambda'T(x)}/\int \e^{\lambda'T(x)}\diff x$ is
\begin{align*}
\log \mathcal{L}(\lambda)&=\sum_{n=1}^N\log f_\lambda(x_n)=\sum_{n=1}^N\left[\lambda'T(x_n)-\log\left(\int \e^{\lambda'T(x)}\diff x\right)\right]\\
&=N\left[\lambda'\hat{T}-\log\left(\int \e^{\lambda'T(x)}\diff x\right)\right],
\end{align*}
which is precisely the expression in \eqref{eq:2.2} multiplied by the sample size $N$.
\end{proof}
We should not misinterpret this result such that the optimum information principle is a special case of maximum likelihood, however, for two reasons. First, I showed \cite[Theorem 5]{toda-maxent-bayes} that maximum likelihood is (approximately) implied by the optimum information principle: we should therefore interpret the above result such that a particular optimum information estimator may coincide to the density corresponding to a particular maximum likelihood estimator.  Second, and more importantly, maximum likelihood is valid only if the model contains the truth (\ie, $f_\lambda=f$ for some $\lambda$), but a model is never true, as \cite{box1976} puts ``Since all models are wrong the scientist cannot obtain a ``correct" one by excessive elaboration".  Hence the maximum likelihood here is actually the quasi-maximum likelihood \cite[]{huber1967}.  On the other hand, from a Bayesian point of view the optimum information principle makes no reference to the truth.

Let us summarize the above results in a theorem and an algorithm to compute the optimum information estimator:
\begin{thm}\label{thm:OIE}
Let $X_n\sim f$, i.i.d.~with $f$ compactly supported on $S\subset\R^K$, $\set{x_n}$ their realizations, $T:S\to \R^L$ be Lebesgue measurable,\\ $\hat{T}:=\frac{1}{N}\sum_{n=1}^NT(x_n)$, $\hat{T}\in\interior T(S)$, and $\dim T(S)=L$.  Then
\begin{enumerate}
\item there exists a unique optimum information estimator $\hat{f}$ defined by \eqref{eq:1},
\item $\hat{f}(x)\propto \e^{\hat{\lambda}'T(x)}$, where $\hat{\lambda}$ is the maximum likelihood estimator of the exponential family $f(x;\lambda)\propto \e^{\lambda'T(x)}$,
\item $\hat{\lambda}$ can be computed by the Newton-Raphson algorithm.
\end{enumerate}
\end{thm}

\begin{algorithm}[Computation of the optimal information estimator]
\noindent
\begin{enumerate}[Step 1.]
\item\label{item:step.1} Choose the relevant support $S$ and moments $T:S\to\R^L$ to exploit.  Choose $S=B(R)$ with $R>\max_n\norm{x_n}$ and $T(x)=(x^\alpha)_{\abs{\alpha}\le A}$, where $A$ is typically 10, unless there is a strong reason to do otherwise.
\item\label{item:step.2} For each $A=2,4,\dots,A$, compute the optimal information estimator by the Newton-Raphson algorithm (Theorem \ref{thm:OIE}).
\item\label{item:step.3} Since the optimum information estimator corresponds to the maximum likelihood distribution of an exponential family (Proposition \ref{prop:ML}), it is natural to use BIC \cite[]{schwarz1978} to select the best estimate among $A=2,4,\dots,A$.  This is the final optimum information estimator.
\end{enumerate}
\end{algorithm}
It is acceptable to simplify Step \ref{item:step.2} by estimating only one density (corresponding to $A$) and omitting Step \ref{item:step.3}.  (More discussion on BIC is given in Section \ref{sec:model}.)

\paragraph{Estimating conditional densities}
Economists are usually interested in the conditional density of a variable $Y$ (\eg, income) conditional on some other variables $X$ (\eg, sex, age, education, experience, etc.) or the conditional expectation $\E[Y|X]$ rather than the unconditional distribution.  Estimation of these quantities are straightforward. For instance, the conditional density of $Y$ given $X$ is estimated by $\hat{f}(y|x):=\frac{\hat{f}(x,y)}{\hat{f}(x)}$, and the conditional expectation is estimated by $\hat{\E}[Y|X]:=\int y \hat{f}(y|x)\diff y$.

\paragraph{Asymptotic properties of the optimum information estimator}
Since the quasi maximum likelihood estimator is a special case of an $M$-estimator \cite[Chapter 5]{vaart1998}, as the sample size $N$ gets large, $\hat{\lambda}$ converges in probability to the $\lambda$ that solves the population counterpart of \eqref{eq:2.2}\footnote{This quantity is known as the \emph{pseudo-true value} \cite[p.~146]{cameron-trivedi2005}.}, and so does the Kullback-Leibler information:
$$\int f\log\frac{f}{\hat{f}}=:H(f;\hat{f})\pto H(f;f_\lambda):=\int f\log\frac{f}{f_\lambda}.$$
The quantity $2NH(f;\hat{f})$ is asymptotically distributed as noncentral $\chi^2$ with $L$ degrees of freedom and noncentrality parameter $2NH(f;f_\lambda)$ \cite[pp.~97--107]{kullback1959}.  Now we show that the optimum information estimator $\hat{f}$ asymptotically approximates the true distribution $f$ arbitrarily well, but we need a lemma first.  To avoid unnecessary complication I assume that the true density $f$ is positive everywhere on its support.

\begin{lem}\label{lem:approx}
Let $f$ be a positive, continuous density on the compact set $S\subset \R^K$.  Then there exists $A>0$ such that the exponential family $f(x;\lambda)\propto \e^{\lambda'T(x)}$, where $T(x)=(x^\alpha)_{\abs{\alpha}\le A}$, contains a density $f_0$ that arbitrarily approximates $f$ uniformly over $S$.
\end{lem}

\begin{proof}
Since $f$ is positive and continuous on $S$, it has a positive minimum and $\log f$ is continuous on $S$.  By the Stone-Weierstrass theorem \cite[p.~139]{folland1999}, we can take $\lambda\in\R^L$ and $C>0$ such that $\abs{\log f(x)-\lambda'T(x)-\log C}$ is arbitrarily small on $S$.  Then we can make $\norm{f-C\e^{\lambda'T}}$ small.  By normalizing the density, we obtain $\abs{f(x)-f_0(x)}<\epsilon$ for all $x\in S$ for some $f_0$ in the exponential family.
\end{proof}

\begin{thm}\label{thm:approx}
Let $f$ be a positive, continuous density on the compact set $S\subset \R^K$.  For any $\epsilon>0$, there exists a number $A>0$ such that for any $\delta>0$, the optimum information estimator $\hat{f}$ corresponding to the moments $T(x)=(x^\alpha)_{\abs{\alpha}\le A}$ satisfies
$$\lim_{N\to\infty}\Pr\left(H(f;\hat{f})>\epsilon+\delta\right)=0.$$
\end{thm}

\begin{proof}
Let $\set{f_\lambda}$ be exponential family in Lemma \ref{lem:approx} and $\hat{f}$ be the optimum information estimator of $f$ using the moments $\abs{\alpha}\le A$.  Take $f_0\in \set{f_\lambda}$ that uniformly approximates $f$.  Since $H(f;\hat{f})\pto H(f;f_\lambda)$, where $f_\lambda$ minimizes the Kullback-Leibler information among the exponential family $\set{f_\lambda}$, we get $H(f;f_\lambda)\le H(f;f_0)$.  Hence we only need to show that we can choose $A$ such that $H(f;f_0)<\epsilon$.  Since $f, f_0$ are positive on $S$, we get
\begin{align*}
H(f;f_0)&=-\int f\log\frac{f_0}{f}=-\int f\log\left(1+\frac{f_0-f}{f}\right)\\
&=-\int f\left[\frac{f_0-f}{f}-\frac{1}{2}\left(\frac{f_0-f}{f}\right)^2+o((f_0-f)^2)\right]\\
&=\int\left[\frac{1}{2f}(f_0-f)^2+o((f_0-f)^2)\right],
\end{align*}
which can be made arbitrarily small by Lemma \ref{lem:approx}.
\end{proof}

\section{Goodness of Fit and Model Selection}\label{sec:model}

This section is an application of the optimum information estimator to evaluate the goodness of fit of parametric models and select the best fitting model.  I consider two cases separately, models for the unconditional density and the conditional density.
\paragraph{Unconditional models}
Suppose that we have $M$ competing models denoted by $\mathcal{M}=\set{1,2,\dotsc,M}$, where model $m$ has a parametric density $f_m(x;\theta_m)$ with parameter $\theta_m\in\Theta_m$.  Given data, how should we choose between these models, and how should we evaluate the goodness of fit?

Ideally, by the optimum information principle the goodness of fit should be evaluated by the minimized Kullback-Leibler information,
\begin{equation}
H_{\min}:=\min_{\theta\in\Theta}\int f(x)\log\frac{f(x)}{f(x;\theta)}\diff x,\label{eq:model.1}
\end{equation}
where $\set{f(x;\theta)}_{\theta\in\Theta}$ is a particular parametric model \cite[Theorem 5]{toda-maxent-bayes}.  However, this approach is infeasible because we do not know the true $f$.  As an approximation, in his seminal paper \cite{akaike1974} approximated the quantity $\int f\log f(x;\hat{\theta})$, which appears in the expansion of \eqref{eq:model.1} and $\hat{\theta}$ denotes the optimum parameter value, and derived his celebrated Akaike Information Criterion (AIC).

We take a different approach.  Since the optimum information principle implies Bayesian inference \cite[Theorem 4]{toda-maxent-bayes}, which is an exact implication as opposed to the approximate implication for maximum likelihood \cite[Theorem 5]{toda-maxent-bayes}, the Bayesian approach to model selection by \cite{schwarz1978} can be fully justified from an information-theoretic point of view.  Now the optimum information estimator $\hat{f}$ in \eqref{eq:2.4}, being a maximum likelihood distribution of a particular exponential family (Theorem \ref{thm:OIE}), is also a particular parametric model.  Therefore the goodness of fit and model selection of the competing models $\mathcal{M}=\set{1,2,\dotsc,M}$ reduces to the model selection of $\set{0}\cup\mathcal{M}$, where model 0 corresponds to the exponential family in Proposition \ref{prop:ML} that generates the optimum information estimator \eqref{eq:2.4}.  Model 0 serves as our benchmark model.

An approximation of the \emph{evidence}\footnote{This term is due to \cite{jaynes1956}.} (the logarithm of the Bayesian likelihood) of model $m$ is given by \cite[p.~461]{schwarz1978}\footnote{Strictly speaking, \cite{schwarz1978} proved \eqref{eq:model.2} only for the exponential family, but since any continuous, compactly supported density can be arbitrarily approximated by an exponential family (Lemma \ref{lem:approx}), \eqref{eq:model.2} is true for any such density.}
\begin{equation}
E_m:=\log\mathcal{L}_m(\hat{\theta}_m)-\frac{1}{2}K_m\log N,\label{eq:model.2}
\end{equation}
where $\mathcal{L}_m$ is the log likelihood of model $m$, $\hat{\theta}_m$ the maximum likelihood estimator, $K_m$ the number of unknown parameters (the dimension of $\theta_m$), and $N$ the sample size.\footnote{I was tempted to call the quantity in \eqref{eq:model.2} TIC, for the obvious reason, but this acronym is reserved for the Takeuchi Information Criterion \cite[]{takeuchi1976}.  Besides, the evidence is $-1/2$ of BIC and hence there is no reason to introduce a new name.} In particular, for the benchmark model 0, we have $K_0=\dim T(S)=L$.  The larger the evidence $E$ is, the better the model fits.

By Laplace's Principle of Indifference \cite[]{laplace1812} (which is a particular axiom of plausible reasoning in \cite{toda-maxent-bayes}), let us assign prior probability $\frac{1}{M+1}$ to models $m=0,1,\dotsc,M$.  Then, by Schwarz's fundamental proposition \cite[p.~462]{schwarz1978} and Bayes's rule (which is implied by the optimum information principle \cite[Theorem 4]{toda-maxent-bayes}), the posterior probability of model $m$ given data $D=\set{x_n}$ is approximately given by
\begin{equation}
P(m|D)=\frac{\e^{NE_m}}{\sum_{m=0}^M\e^{NE_m}}.\label{eq:model.3}
\end{equation}
$P(m|D)$ is a measure of model fit.  If $P(m|D)$ is large for some $m=1,2,\dots,M$, then it is a good model.  If $P(0|D)$ is large, then all models are poor.  The fundamental difference between our evidence $E$ and the posterior model probability $P(m|D)$ and other information criteria such as AIC and BIC is that while AIC and BIC are \emph{relative} measures of model fit, evidence $E$ and $P(m|D)$ are \emph{absolute} measures.  By using AIC and BIC we can select the best model among the competing models, but it might be the case that all models are poor.  On the other hand, our approach tells us each model's absolute fit.

\paragraph{Conditional models}
Suppose now that the models are conditional, meaning that they only describe a conditional density $f(y|x;\theta)$, which is more important in economics.  The optimum information principle gives us a measure of goodness of fit and model selection in this case, too.  The evidence $E_m$ in \eqref{eq:model.2} is changed to
\begin{equation}
E_m:=\log\mathcal{L}_m(\hat{\theta}_m;Y|X)+\log\mathcal{L}(X)-\frac{1}{2}K_m\log N,\label{eq:model.4}
\end{equation}
where $\log\mathcal{L}_m(\hat{\theta}_m;Y|X)$ denotes the maximized conditional log likelihood of model $m$ and $\log\mathcal{L}(X)$ denotes the log likelihood of data $X=\set{x_n}$.  Since the true density of $\set{X_n}$ is unknown, $\log\mathcal{L}(X)$ is a nuisance parameter.  But since it is common across all models, if we substitute $E_m$ in \eqref{eq:model.4} into the fundamental formula \eqref{eq:model.3}, $\exp(\log\mathcal{L}(X))=\mathcal{L}(X)$ in the numerator and denominator cancel out (!).  Therefore the posterior model probability $P(m|D)$ can just be computed by using the maximized conditional log likelihood.  Let us summarize this in a theorem:

\begin{thm}[Goodness of fit and model selection]\label{thm:model}
Let $\mathcal{M}=\set{0,1,2,\dotsc,M}$ be $M$ competing models with parametric conditional density $f_m(y|x;\theta_m)$, where model 0 is the benchmark model corresponding to the optimum information estimator $\hat{f}(y|x):=\frac{\hat{f}(x,y)}{\hat{f}(x)}$.  Let $\hat{\theta}_m$ be the maximum likelihood estimator of model $m$ and define the conditional evidence of model $m>0$ by
$$E_m:=\log\mathcal{L}_m(\hat{\theta}_m;Y|X)-\frac{1}{2}K_m\log N,$$
where $K_m=\dim\theta_m$ and $N$ is the sample size, and
$$E_0:=\log\mathcal{L}(\hat{\theta}_{X,Y};X,Y)-\log\mathcal{L}(\hat{\theta}_X;X)
-\frac{1}{2}(K_{X,Y}-K_X)\log N,$$
where $\log \mathcal{L}(\hat{\theta}_Z;Z)$ denotes the maximized log likelihood of the exponential family corresponding to the optimum information estimator for the density $f(z)$ ($z=x$ or $z=(x,y)$), and $K_Z$ is its parameter dimension.  Then, the posterior probability of model $m$ is
$$P(m|D)=\frac{\e^{NE_m}}{\sum_{m=0}^M\e^{NE_m}}.$$
\end{thm}

Our information-theoretic approach to the goodness of fit of conditional distribution is much simpler than the frequentist approach such as \cite{andrews1997}, \cite{fan1998}, and \cite{delgado-stute2008} since it requires no bootstrapping, no kernel density estimation, or no complicated integration.  All we need is maximum likelihood estimation.  Another advantage is that the frequentist approach can only test a null hypothesis, thereby accepting or rejecting a model, but our approach can evaluate an arbitrary number of models simultaneously.  A related method to our approach of model comparison is the likelihood ratio test (see \cite{wilks1938} for nested models and \cite{vuong1989} for non-nested models), but it only applies to \emph{two} models and it provides no information for the goodness of fit.  The information-theoretic approach is applicable to any number of non-nested models and gives an absolute measure of goodness of fit.

\paragraph{AIC or BIC?}
I used BIC \cite[]{schwarz1978} to define the evidence of a model in \eqref{eq:model.2}.  However, there are other information criteria such as AIC, AICc, BIC, CIC, DIC, EIC, QIC, TIC, etc.~(see \cite{burnham-anderson2004a} and \cite{anderson2008}), among which AIC (AICc) and BIC are by far the most applied.  Anderson \cite[p.~160]{anderson2008} dismisses BIC as having ``nothing linking it to information theory", which is incorrect, but Anderson's book was written before my discovery \cite[]{toda-maxent-bayes}.  I believe that BIC is the most fundamental concept because Bayesian inference is an exact implication of the optimum information principle and \cite{schwarz1978} uses only one approximation to derive BIC (approximating log likelihood), whereas the approach of \cite{akaike1974} is conceptually closer to the optimum information principle but it invokes \emph{two} approximations (approximating the Kullback-Leibler information and log likelihood).  But let us not be dogmatic: it is equally acceptable to use AIC or AICc \cite[]{sugiura1978}, which are also derived by information theory.

\section{Performance of OIE with Real Data}
To the best of my knowledge, the only applications of the maximum entropy estimation in economics and finance are \cite{wu2003} and the references therein.  The estimated maximum entropy density superimposed on the histogram of 1999 U.S.~family income in Figure 1 of \cite{wu2003} shows a good fit, as predicted by the theory (Theorem \ref{thm:approx}).  Maximum entropy estimation (information-theoretic estimation) is much more popular outside economics, in particular physics.  Jaynes \cite[p.~125]{jaynes2003} mentions that the Bayesian analysis (synonymous to information-theoretic analysis) of nuclear magnetic resonance (NMR)\footnote{NMR is applied in medicine for making 2D and 3D images of the inside of the human body for diagnostic purposes, which is known as magnetic resonance imaging (MRI).  (``Nuclear" is dropped because it is not a politically correct word.)} data obtained by his student \cite[]{bretthorst1988} showed an orders of magnitude (\ie, at least 10 fold) improvement of resolution over Fourier transforms method which was conventional at the time, and because of this surprising improvement Bretthorst's result was not believed for a long time.

The value of the information-theoretic nonparametric estimation (and model selection) method proposed in this paper compared to conventional methods such as kernel density estimation should ultimately be judged by their relative performances in analyzing real data.  However, there are a few reasons to believe that the optimum information estimator is superior:
\begin{enumerate}
\item The optimum information principle fully exploits the available information as opposed to other \emph{ad hoc} methods.  For instance, kernel density estimation is essentially a local linear regression and hence uses only \emph{local} information.
\item Kernel density estimation has a lot of arbitrariness with regard to the choice of the kernel and the bandwidth, whereas the only arbitrariness in the information-theoretic density estimation is the number of moments to include as constraints.  Even this arbitrariness can be removed by selecting the optimal number of constraints by BIC.
\item Since the optimum information estimation reduces to the maximum likelihood estimation of an exponential family (in the present case), it is fully parametric, computationally straightforward, and free from the curse of dimensionality.
\end{enumerate}

\section{Concluding Remarks}
Statistics and econometrics are sciences of extracting information from data.  Hence an inference method is valuable if and only if it is useful in analyzing real data, and therefore an inference method requires no interpretation, and no justification except practical usefulness: we should refrain from being too dogmatic as exemplified by the heated frequentist/Bayesian debate in the past.  Comparisons of the performance of the optimum information estimator and other methods using real data are welcome, although it is beyond the scope of the present paper.


\end{document}